\documentclass[12pt, reqno]{amsart}
\makeatletter
\@namedef{subjclassname@1991}{$\mathrm{1991}$ Mathematics Subject Classification}
\@namedef{subjclassname@2000}{$\mathrm{2000}$ Mathematics Subject Classification}
\@namedef{subjclassname@2010}{$\mathrm{2010}$ Mathematics Subject Classification}
\@namedef{subjclassname@2020}{$\mathrm{2020}$ Mathematics Subject Classification}
\makeatother
\usepackage{amsmath,amsthm, amscd, amsfonts, amssymb, graphicx, color}
\usepackage[bookmarksnumbered, colorlinks, plainpages,linkcolor=blue,urlcolor=blue,citecolor=blue]{hyperref}
\newtheorem{thm}{Theorem}[section]
\newtheorem{cor}[thm]{Corollary}
\newtheorem{lem}[thm]{Lemma}

\newtheorem{defn}[thm]{Definition}
\newtheorem{rem}[thm]{\bf{Remark}}

\numberwithin{equation}{section}

\begin{document}

\title{generalized weighted EP elements in Banach algebras}

\author{Huanyin Chen}
\author{Marjan Sheibani}
\address{School of Big Data, Fuzhou University of International Studies and Trade, Fuzhou 350202, China}
\email{<huanyinchenfz@163.com>}
\address{Farzanegan Campus, Semnan University, Semnan, Iran}
\email{<m.sheibani@semnan.ac.ir>}


\begin{abstract} We propose a new class of generalized inverses with weights, which represent a natural extension of EP (Moore-Penrose) and *-DMP (Drazin-Moore-Penrose) elements in a Banach *-algebra. This paper presents various characteristics of weighted EP elements. Moreover, we characterize the weighted EP element through the core-EP decomposition and a polar-like property. Finally, we explore weighted *-DMP elements and uncover several new properties of *-DMP elements.\end{abstract}

\maketitle

\section{Introduction}

Let $\mathcal{A}$ be a Banach algebra with involution $*$. An element $a$ in $\mathcal{A}$ is an EP element if there exists some $x\in \mathcal{A}$ such that
$$ax^2=x, xa^2=a, (ax)^*=ax, (xa)^*=xa.$$ Let ${\Bbb C}^{n\times n}$ be the Banach algebra of all $n\times n$ complex matrices with conjugate transpose $*$.
$A\in {\Bbb C}^{n\times n}$ is an EP matrix if and only if $rank(A)=rank(A^*)$. Evidently, $a$ is an EP element if and only if there exists some $x\in \mathcal{A}$ such that $ax^2=x, xa^2=a, (ax)^*=xa$ if and only if there exists $x\in \mathcal{A}$ such that $a^2x=a, ax=xa, (ax)^*=ax$ if and only if $a\in \mathcal{A}^{\#}$ and $(aa^{\#})^*=aa^{\#}$ (~\cite{B, H, K,M2,W,X,Z}).
Here,  $a\in \mathcal{A}$ has group inverse provided that there exists $x\in \mathcal{A}$ such that $$ax^2=x, ax=xa, a=xa^2.$$ Such $x$ is unique if exists, denoted by $a^{\#}$, and called the group inverse of $a$.

An element $a$ in a Banach *-algebra $\mathcal{A}$ is a *-DMP element if there exist $m\in {\Bbb N}$ and $x\in \mathcal{A}$ such that
$ax^2=x, (ax)^*=ax, (xa)^*=xa, a^n=xa^{n+1}$. As is well known, $a\in \mathcal{A}$ is *-DMP if and only if $a^m\in \mathcal{A}$ is EP for some $m\in {\Bbb N}$ (see~\cite{C3,G2,G3,GZ}). It was proved tht $a\in \mathcal{A}$ has *-DMP if and only if there exist $n\in {\Bbb N}$ and $x\in \mathcal{A}$ such that
$ax^2=x, (ax)^*=xa, a^m=xa^{m+1}.$  We propose a novel class of generalized inverses with weights, which represent a natural extension of EP (Moore-Penrose) and *-DMP (Drazin-Moore-Penrose) elements in a Banach *-algebra.

\begin{defn} An element $a\in \mathcal{A}$ is a $w$-EP element if there exists $x\in \mathcal{A}$ such that
$$a(wx)^2=x, x(wa)^2=a, (awxw)^*=awxw, (xwaw)^*=xwaw.$$\end{defn}

In Section 2, we present various characterizations of weighted EP elements. Let $\mathcal{A}_w^{\tiny{e}}$ be the set of all $w$-EP elements in $\mathcal{A}$. Let $\mathcal{A}^{qnil}=\{x\in \mathcal{A}~\mid~ \lim\limits_{n\to \infty}\parallel x^n\parallel^{\frac{1}{n}}=0\}.$ Evidently,
$x\in \mathcal{A}^{qnil}$ if and only if $1+\lambda x\in \mathcal{A}^{-1}$ for any $\lambda\in {\Bbb C}$. Set $\mathcal{A}_w^{qnil}
=\{y\in \mathcal{A}~\mid~ yw\in \mathcal{A}^{qnil}\}$.

\begin{defn} An element $a\in \mathcal{A}$ is a generalized $w$-EP element if there exists $x\in \mathcal{A}$ such that $$\begin{array}{c}
a(wx)^2=x, (awxw)^*=awxw, (xwaw)^*=xwaw,\\
\lim\limits_{n\to \infty}||(aw)^n-xw(aw)^{n+1}||^{\frac{1}{n}}=0.
\end{array}$$\end{defn}

In Section 3, we provide characterizations of generalized weighted EP elements through the core-EP decomposition and a polar-like property.
Let $\mathcal{A}_w^{\tiny\textcircled{e}}$ be the set of all generalized $w$-EP elements in $\mathcal{A}$.
We prove that $a\in \mathcal{A}_w^{\tiny\textcircled{e}}$ if and only if there exist $x,y\in \mathcal{A}$ such that $$a=x+y, x^*y=ywx=0, x\in \mathcal{A}_w^{\tiny{e}},~y\in \mathcal{A}_w^{qnil}$$ if and only if there exists a projection $p\in \mathcal{A}$ (i.e., $p^2=p=p^*$) such that
$$aw+p\in \mathcal{A}^{-1}, 1-p\in \mathcal{A}w, awp=paw\in \mathcal{A}^{qnil}.$$

In Section 4, we are concerned with the representations of generalized weighted EP elements by using the weighted generalized Drazin inverse.
Here, an element $a\in \mathcal{A}$ has generalized $w$-Drazin inverse $x$ if there exists unique $x\in \mathcal{A}$ such that
 $$awx=xwa, xwawx=x ~\mbox{and}~ a-awxwa\in \mathcal{A}^{qnil}.$$ We denote $x$ by $a^{d,w}$ (see~\cite{M}). Let $\mathcal{A}^{d,w}$ be the set of
 all generalized $w$-Drazin invertible elements in $\mathcal{A}$. We prove that $a\in \mathcal{A}_w^{\tiny\textcircled{e}}$ if and only if $a\in \mathcal{A}^{d,w}$ and $a^{d,w}\in \mathcal{A}_w^{\tiny{e}}$. In this case, $a_w^{\tiny\textcircled{e}}=[a^{d,w}w]^2[a^{d,w}]_w^{\tiny{e}}.$

For further study *-DMP elements, we introduce weighted *-DMP elements in a Banach *-algebra.

\begin{defn} An element $a\in \mathcal{A}$ is a w-*-DMP element if there exists $x\in \mathcal{A}$ such that
$$a(wx)^2=x, (awxw)^*=awxw, (xwaw)^*=xwaw, (aw)^n=xw(aw)^{n+1}$$ for some $n\in {\Bbb N}$.\end{defn}

Finally, in Section 5, weighted *-DMP elements in a Banach *-algebra are characterized. New properties of *-DMP elements are thereby obtained. Many properties of EP elements and *-DMP elements are extended in wiader cases.

\section{Weighted EP elements}

An element $a\in \mathcal{A}$ has $w$-group inverse if there exist $x\in \mathcal{A}$ such that
$$awxwa=a, xwawx=x, awx=xwa.$$ The preceding $x$ is unique if it exists, and we denote it by $a_w^{\#}$. The set of all generalized $w$-group invertible elements in $\mathcal{A}$ is denoted by $\mathcal{A}_w^{\#}$.

\begin{thm} Let $a,w\in \mathcal{A}$. Then the following are equivalent:\end{thm}
\begin{enumerate}
\item [(1)] $a\in \mathcal{A}_w^{\tiny{e}}$.
\vspace{-.5mm}
\item [(2)] There exists $x\in \mathcal{A}$ such that
$$a(wx)^2=x, x(wa)^2=a, (xwaw)^*=xwaw.$$
\vspace{-.5mm}
\item [(3)] $a\in \mathcal{A}_w^{\#}$ and $(awa_w^{\#}w)^*=awa_w^{\#}w$.
\end{enumerate}
\begin{proof} $(1)\Rightarrow (2)$ This is trivial.

$(2)\Rightarrow (3)$ Let $z=(xw)^2a$. Then we check that
$$\begin{array}{rll}
awzwa&=&aw(xw)^2awa=a(wx)^2(wa)^2=x(waw)^2=a,\\
zwawz&=&(xw)^2awaw(xw)^2a=xw[x(wa)^2]wxwxwa\\
&=&xw[a(wx)^2]wa=(xw)^2a=x,\\
awz&=&aw(xw)^2a=[a(wx)^2]wa=xwa\\
&=&xwx(wa)^2=(xw)^2awa=zwa.
\end{array}$$ Hence $a\in \mathcal{A}_w^{\#}$ and
$a_w^{\#}=(xw)^2a$. Thus, $awa_w^{\#}w=aw(xw)^2aw=[a(wx)^2]waw=xwaw$. Therefore $(awa_w^{\#}w)^*=(xwaw)^*=xwaw=awa_w^{\#}w$, as required.

$(3)\Rightarrow (1)$ Let $x=a_w^{\#}$. Then $$awxwa=a, xwawx=x, awx=xwa.$$ Hence $a(wx)^2=(awx)wx=(xwa)wx=x$ and $x(wa)^2=(xwa)wa$ $=(awx)wa=a$.
By hypothesis, $(awxw)^*=awxw$. Moreover, $xwaw=awxw$, and then $(xwaw)^*=xwaw$.
This implies that $x=a_w^{\tiny{e}}$, as required.\end{proof}

\begin{cor} (see~\cite[Theorem 7.3]{KP}) Let $a\in \mathcal{A}$. Then $a\in \mathcal{A}^{\tiny{EP}}$ if and only if\end{cor}
\begin{enumerate}
\item [(1)] $a\in \mathcal{A}^{\#};$
\vspace{-.5mm}
\item [(2)] $(aa^{\#})^*=aa^{\#}$.
\end{enumerate}
\begin{proof} This is obvious by choosing $w=1$ in Theorem 2.1.\end{proof}

We are ready to prove:

\begin{thm} Let $a,w\in \mathcal{A}$. Then the following are equivalent:\end{thm}
\begin{enumerate}
\item [(1)] $a\in \mathcal{A}_w^{\tiny{e}}$.
\vspace{-.5mm}
\item [(2)] There exists $x\in \mathcal{A}$ such that
$$(aw)^2x=a, awx=xwa, (awxw)^*=awxw.$$
\item [(3)] There exists $x\in \mathcal{A}$ such that
$$(aw)^2x=a, (xw)^2a=x, (awxw)^*=awxw.$$
\end{enumerate}
\begin{proof}  $(1)\Rightarrow (2)$ This is trivial by Theorem 2.1.

$(2)\Rightarrow (1)$ Let $z=a(wx)^2$. Since $awx=xwa$, we check that
$$\begin{array}{rll}
awzwa&=&awa(wx)^2wa=[(aw)^2x]wxwa=awxwa=(aw)^2x=a,\\
zwawz&=&a(wx)^2wawa(wx)^2=a(wx)^2wxw[(aw)^2x]\\
&=&a(wx)^2wxwa=[(aw)^2x]wxwx=a(wx)^2=z,\\
awz&=&awa(wx)^2=[(aw)^2x]wx=awx\\
&=&[(aw)^2x]wx=a(wx)^2wa=zwa.
\end{array}$$ Hence $a\in \mathcal{A}_w^{\#}$ and
$a_w^{\#}=a(wx)^2$. Hence, $awa_w^{\#}w=awa(wx)^2w=[(aw)^2x]wxw=awxw$. Therefore $(awa_w^{\#}w)^*=(awxw)^*=awxw=awa_w^{\#}w$.
In light of Theorem 2.1, $a\in \mathcal{A}_w^{\tiny{e}}$.

$(1)\Rightarrow (3)$ Let $x=a_w^{\#}$. In view of Theorem 2.1, $x=a_w^{\tiny{e}}$. By the argument above,
we have $$(aw)^2x=a, awx=xwa, (awxw)^*=awxw.$$
Moreover, we have $(xw)^2a=xw(xwa)=xw(awx)=(xwa)wx=(awx)(wx)=a(wx)^2=x$, as required.

$(3)\Rightarrow (1)$ Set $z=a(wx)^2$. Then we verify that
$$\begin{array}{rll}
awz&=&awa(wx)^2=[(aw)^2x]wx=awx\\
&=&aw[(xw)^2a]=a(wx)^2wa=zwa,\\
(aw)^2z&=&(awz)wz=(zwa)wa(wx)^2=zw[(aw)^2x]wx\\
&=&(zwa)wx=(awx)wx=a(wx)^2=z,\\
awzw&=&(awz)w=(awx)w=awxw,\\
(awzw)^*&=&awzw.
\end{array}$$ This completes the proof by Theorem 2.1.\end{proof}

\begin{cor} (see~\cite[Theorem 2.7 and Theorem 2.11]{X}) Let $a\in \mathcal{A}$. Then the following are equivalent:\end{cor}
\begin{enumerate}
\item [(1)] $a\in \mathcal{A}^{\tiny{EP}}$.
\vspace{-.5mm}
\item [(2)] There exists $x\in \mathcal{A}$ such that
$$a^2x=a, ax=xa,(ax)^*=ax.$$
\vspace{-.5mm}
\item [(3)] There exists $x\in \mathcal{A}$ such that
$$a^2x=a, x^2a=x, (ax)^*=ax.$$
\end{enumerate}
\begin{proof}  This is obvious by choosing $w=1$ in Theorem 2.3.\end{proof}

\begin{cor} Let $a,w\in \mathcal{A}$. Then the following are equivalent:\end{cor}
\begin{enumerate}
\item [(1)] $a\in \mathcal{A}_w^{\tiny{e}}$.
\vspace{-.5mm}
\item [(2)] $a\in \mathcal{A}_w^{\#}$ and $a\mathcal{A}=(aw)^*\mathcal{A}$.
\end{enumerate}
\begin{proof}  $(1)\Rightarrow (2)$ In view of Theorem 2.1, $a\in \mathcal{A}_w^{\#}, x(wa)^2=a$ and $(xwaw)^*=xwaw$ for some $x\in \mathcal{A}$.
Hence, $a=(xwaw)a=(xwaw)^*a=(aw)^*(xw)^*a\in (aw)^*\mathcal{A}$. By virtue of Theorem 2.3, there exists some $z\in \mathcal{A}$ such that
$(aw)^2z=a$ and $(awzw)^*=awzw$, and then $aw(awzw)=aw$. hence, $(aw)^*=(awzw)^*(aw)^*=awzw(aw)^*\in a\mathcal{A}$. Therefore $a\mathcal{A}=(aw)^*\mathcal{A}$, as required.

$(2)\Rightarrow (1)$ Set $x=a_w^{\#}$. Then $a=awxwa$ and $awx=xwa$; and so
$(1-awxw)a=0$. Since $a\mathcal{A}=(aw)^*\mathcal{A}$, we have $(1-awxw)(aw)^*=0$. Thus $(aw)^*=awxw(aw)^*$.
We infer that $(xwaw)^*=(aw)^*(xw)^*=awxw(aw)^*(xw)^*=awxw(xwaw)^*$; hence,
$(awxw)^*=awxw(awxw)^*$. Therefore $awxw=(awxw)(awxw)^*=(awxw)^*$. Accordingly, $a\in \mathcal{A}_w^{\tiny{e}}$ by Theorem 2.3.
\end{proof}

\begin{thm} Let $a,w\in \mathcal{A}$. Then the following are equivalent:\end{thm}
\begin{enumerate}
\item [(1)] $a\in \mathcal{A}_w^{\tiny{e}}$.
\vspace{-.5mm}
\item [(2)] There exists $x\in \mathcal{A}$ such that
$$x(wa)^2=a, (awxw)^*=xwaw.$$
\vspace{-.5mm}
\item [(3)] There exists $x\in \mathcal{A}$ such that
$$(aw)^2x=a, (awxw)^*=xwaw.$$
\end{enumerate}
\begin{proof} $(1)\Rightarrow (2)$ Let $x=a_w^{\tiny{e}}$. In view of Theorem 2.3, we directly verify that $$x(wa)^2=a, (awxw)^*=(awx)w=(xwa)w=xwaw,$$ as desired.

$(2)\Rightarrow (1)$ By hypothesis, there exists $x\in \mathcal{A}$ such that
$$x(wa)^2=a, (awxw)^*=xwaw.$$ Then $$\begin{array}{rll}
xwaw&=&(awxw)^*=[x(wa)^2wxw]^*[(xwaw)(awxw)]^*\\
&=&(awxw)^*(xwaw)^*=(xwaw)(awxw)\\
&=&x(wa)^2wxw=awxw.
\end{array}$$ Hence $(awxw)^*=xwaw=awxw$. This implies that
$$a=x(wa)^2=(xwaw)a=(awxw)^*a=awxwa.$$

Let $z=xwawx$. Then $$\begin{array}{rll}
zwawz&=&xw(awxwa)wxwawx=xw(awxwa)wx=xwawx=z,\\
z(wa)^2&=&xwaw[x(wa)^2]=x(wa)^2=a,\\
zwaw&=&xw(awxwa)w=xwaw=(awxw)^*,\\
(zwaw)^*&=&awxw=(awxw)^*=zwaw,\\
a(wz)^2&=&awzwz=(awxwa)wxwxwawx=(aw)(xw)xwawx\\
&=&(xw)(aw)xwawx=xw(awxwa)wx=xwawx=z.
\end{array}$$ Therefore $a\in \mathcal{A}_w^{\tiny{e}}$ by Theorem 2.1.

$(1)\Longleftrightarrow (3)$ This is proved in the same manner.\end{proof}

\begin{cor} Let $a,w\in \mathcal{A}$. Then the following are equivalent:\end{cor}
\begin{enumerate}
\item [(1)] $a\in \mathcal{A}_w^{\tiny{e}}$.
\vspace{-.5mm}
\item [(2)] There exists $x\in \mathcal{A}$ such that
$$a(wx)^2=x, awxwa=a, (awxw)^*=xwaw.$$
\vspace{-.5mm}
\item [(3)] There exists $x\in \mathcal{A}$ such that
$$(xw)^2a=x, awxwa=a, (awxw)^*=xwaw.$$
\end{enumerate}
\begin{proof} This is proved by Theorem 2.3 and Theorem 2.6.\end{proof}

\begin{cor} Let $a,w\in \mathcal{A}$. Then the following are equivalent:\end{cor}
\begin{enumerate}
\item [(1)] $a\in \mathcal{A}_w^{\tiny{e}}$.
\vspace{-.5mm}
\item [(2)] There exists $x\in \mathcal{A}$ such that
$$a^2x=a, (ax)^*=xa.$$
\vspace{-.5mm}
\item [(3)] There exists $x\in \mathcal{A}$ such that
$$xa^2=a, (ax)^*=xa.$$
\vspace{-.5mm}
\item [(4)] There exists $x\in \mathcal{A}$ such that
$$ax^2=x, axa=a, (ax)^*=xa.$$
\vspace{-.5mm}
\item [(5)] There exists $x\in \mathcal{A}$ such that
$$x^2a=a, axa=a, (ax)^*=xa.$$
\end{enumerate}
\begin{proof} This is obvious by Theorem 2.6 and Corollary 2.7.\end{proof}

\begin{lem} Let $a\in \mathcal{A}_w^{\tiny{e}}$. Then $xwawx=x$.\end{lem}
\begin{proof}  By hypothesis, we can find $x\in \mathcal{A}$ such that
$$a(wx)^2=x, x(wa)^2=a, (awxw)^*=awxw, (xwaw)^*=xwaw.$$
Then $$\begin{array}{rll}
xwawx&=&xwawa(wx)^2=x(wa)^2(wx)^2\\
&=&a(wx)^2=x.
\end{array}$$\end{proof}

We use $\ell(x)$ and $r(x)$ to denote the left and right annihilator of $x$ in $\mathcal{A}$, respectively. We derive

\begin{thm} Let $a,w\in \mathcal{A}$. Then the following are equivalent:\end{thm}
\begin{enumerate}
\item [(1)] $a\in \mathcal{A}_w^{\tiny{e}}$.
\vspace{-.5mm}
\item [(2)] There exists $x\in \mathcal{A}$ such that
$$awxwa=a, x\mathcal{A}=aw\mathcal{A}, \mathcal{A}x^*=\mathcal{A}aw.$$
\item [(3)] There exists $x\in \mathcal{A}$ such that
$$xwawx=x, x\mathcal{A}=aw\mathcal{A}, \mathcal{A}x^*=\mathcal{A}aw.$$
\item [(4)] There exists $x\in \mathcal{A}$ such that
$$awxwa=a, \ell(x)=\ell(aw), r(x^*)=r(aw).$$
\item [(5)] There exists $x\in \mathcal{A}$ such that
$$xwawx=x, \ell(x)=\ell(aw), r(x^*)=r(aw).$$
\end{enumerate}
\begin{proof}  $(1)\Rightarrow (2)$ Let $x=a_w^{\tiny{e}}$. Then $awxwa=a$.
Moreover, we have $$a(wx)^2=x, x(wa)^2=a, (xwaw)^*=xwaw.$$ Hence, $x=aw(xwx)\in \mathcal{A}$ and $aw=x(wa)^2w\in \mathcal{A}$. This implies that
$x\mathcal{A}=aw\mathcal{A}$. In view of Lemma 2.9, $x=xwawx$, and then $x^*=x^*(xwaw)^*=x^*xwaw\in \mathcal{A}aw$.
Since $aw=awxwaw$, we have $aw=aw(xwaw)=aw(xwaw)^*=aw(waw)^*x^*\in \mathcal{A}x^*$. Therefore
$\mathcal{A}x^*=\mathcal{A}aw$.

$(2)\Rightarrow (4)$ If $rx=0$ for $r\in \mathcal{A}$, then $raw=0$. If $raw=0$ for $r\in \mathcal{A}$, then $rx=0$.
Hence $\ell(x)=\ell(aw)$. Likewise, we have $r(x^*)=r(aw),$ as required.

$(4)\Rightarrow (1)$ Since $awxwa=a$, we have $1-awxw\in \ell(aw)$, and then $1-awxw\in \ell(x)$.
This implies that $x=a(wx)^2$. Hence, $x^*=x^*(awxw)^*$. This implies that $1-(awxw)^*\in r(x^*)$, and then
$1-(awxw)^*\in r(aw)$. We infer that $aw=(awxw)^*aw$; whence, $awxw=(awxw)^*awxw$. Therefore
$(awxw)^*=(awxw)^*awxw=awxw$. As $1-xwaw\in r(aw)$, we have $1-xwaw\in r(x^*)$. it follows that
$x^*=x^*xwaw$. Then $(waw)^*x^*=(waw)^*x^*xwaw$, i.e., $(xwaw)^*=(xwaw)^*xwaw$. Thus $(xwaw)^*=xwaw$.

Since $a=awxwa$, we have $aw=awxwaw$, and so $1-xwaw\in r(aw)$.
Then $1-xwaw\in r(x^*)$. This shows that $x^*=x^*(xwaw)$; whence, $x=(xwaw)^*x=xwawx$.
Thus $1-xwaw\in \ell(x)=\ell(aw)$. This implies that $aw=xw(aw)^2$, and then
$$a=(aw)xwa=xw(aw)^2xwa=xwaw(awxwa)=xwawa=x(wa)^2.$$ Therefore
$a\in \mathcal{A}_w^{\tiny{e}}$.

$(1)\Rightarrow (3)$ By virtue of Lemma 2.9, $x=xwawx$. As in the argument above, we prove that $x\mathcal{A}=aw\mathcal{A}, \mathcal{A}x^*=\mathcal{A}aw,$ as required.

$(3)\Rightarrow (5)$ This is obvious.

$(5)\Rightarrow (1)$ This is similar to the discussion in $(3)\Rightarrow (1)$.\end{proof}

\begin{cor} Let $a\in \mathcal{A},w\in \mathcal{A}^{-1}$. Then the following are equivalent:\end{cor}
\begin{enumerate}
\item [(1)] $a\in \mathcal{A}_w^{\tiny{e}}$.
\vspace{-.5mm}
\item [(2)] $a\in \mathcal{A}_w^{\tiny\textcircled{\#}}$ and there exists $u\in \mathcal{A}^{-1}$ such that
$a_w^{\tiny\textcircled{\#}}=awu$
\end{enumerate}
\begin{proof}  $(1)\Rightarrow (2)$ Set $u=[(xw)^2+1-xwaw]w^{-1}$. Then $u^{-1}=w[(aw)^2+1-xwaw]$.
Moreover, we check that $awu=aw[(xw)^2+1-xwaw]w^{-1}=a(wx)^2=x$, as required.

$(2)\Rightarrow (1)$ Set $x=a_w^{\tiny\textcircled{\#}}$. Then $awxwa=a, x=xwawx, (awxw)^*=awxw$ and $awx=xwa$. Hence, $x=a(wx)^2$ and $aw=x(wa)^2w$.
This implies that $x\mathcal{A}=aw\mathcal{A}$. Clearly, $aw=awxwaw=aw(xwaw)^*$ $=aw(waw)^*x^*$. Moreover, we have
$x^*=(awu)^*=u^*(awxwaw)^*=u^*(aw)^*(awxw)^*=u^*(aw)^*(awxw)=[u^*(aw)^*(xw)](aw)$.
Hence, $\mathcal{A}x^*=\mathcal{A}aw$. Therefore $a\in \mathcal{A}_w^{\tiny{e}}$ by Theorem 2.10.\end{proof}

\section{generalized weighted EP elements}

The aim of this section is to introduce generalized weighted EP elements in a Banach algebra. Many elementary properties of this new generalized inverse are presented.

\begin{thm} Let $a\in \mathcal{A}$. Then the following are equivalent:\end{thm}
\begin{enumerate}
\item [(1)] $a\in \mathcal{A}_w^{\tiny\textcircled{e}}$.
\vspace{-.5mm}
\item [(2)] There exist $x,y\in \mathcal{A}$ such that $$a=x+y, x^*y=ywx=0, x\in \mathcal{A}_w^{\tiny{e}},~y\in \mathcal{A}_w^{qnil}.$$
\end{enumerate}
\begin{proof} $(1)\Rightarrow (2)$ By hypothesis, we can find $x\in \mathcal{A}$ such that
$$\begin{array}{c}
a(wx)^2=x, (awxw)^*=awxw, (xwaw)^*=xwaw,\\
\lim\limits_{n\to \infty}||(aw)^n-xw(aw)^{n+1}||^{\frac{1}{n}}=0.
\end{array}$$ Let $z=awxwa$ and $y=a-awxwa$. Then we check that
$$\begin{array}{rll}
z^*y&=&(awxwa)^*(a-awxwa)=a^*(awxw)^*(a-awxwa)\\
&=&a^*awxw(a-awxwa)=a^*[awxwa-aw(xwawx)wa]\\
&=&a^*[awxwa-awxwa]=0,\\
ywz&=&(a-awxwa)wawxwa=(aw)^2xwa-aw[xw(aw)^2xw]a\\
&=&(aw)^2xwa-aw[(aw)xw]a=0,\\
yw&=&(a-awxwa)w=aw-awxwaw,\\
\end{array}$$
$$\begin{array}{rl}
&\big(aw-xw(aw)^2\big)^2\\
=&[aw-xw(aw)^2][aw-xw(aw)^2]=[aw-xw(aw)^2]aw,\\
&\big(aw-xw(aw)^2\big)^k\\
=&[aw-xw(aw)^2](aw)^{k-1}=(aw)^k-(xw)(aw)^{k+1},\\
yw&\in\mathcal{A}_w^{qnil}.
\end{array}$$
We claim that $z_w^{\tiny{e}}=x.$ It is easy to verify that

$$\begin{array}{rll}
z(wx)^2&=&aw(xwawx)wx=a(wx)^2=x,\\
x(wz)^2&=&(xwawx)wawawxwa=xw(aw)^2xwa=awxwa=z,\\
zwxw&=&aw(xwawx)w=awxw,\\
(zwxw)^*&=&zwxw,\\
xwzw&=&(xwawxw)aw=xwaw,\\
(xwzw)^*&=&xwzw.
\end{array}$$

$(2)\Rightarrow (1)$ By hypothesis, there exist $z,y\in \mathcal{A}$ such that $a=z+y, z^*y=ywz=0, z\in \mathcal{A}_w^{\tiny{e}},
y\in \mathcal{A}^{qnil}.$ Set $x=z_w^{\tiny{e}}$. In view of Lemma 2.9,
$xwzwx=x.$ Moreover, we check that
$$\begin{array}{rll}
a(wx)^2&=&(z+y)[wz_w^{\tiny{e}}]^2=z[wz_w^{\tiny{e}}]^2=z_w^{\tiny{e}}=x,\\
awxw&=&(y+z)wz_w^{\tiny{e}}w=zwz_w^{\tiny{e}}w,\\
(awxw)^*&=&awxw,\\
xwaw&=&xwzwxwaw=xw[zwz_w^{\tiny{e}}w]^*(z+y)w\\
&=&xw[wz_w^{\tiny{e}}w]^*z^*(z+y)w=xw[wz_w^{\tiny{e}}w]^*z^*zw\\
&=&xw[zwz_w^{\tiny{e}}w]^*zw=xw[zwxw]zw=xwzw,\\
(xwaw)^*&=&xwaw.
\end{array}$$
As in the preceding discussion, we get $(1-xwaw)yw=yw.$ Hence, we prove that
$$\begin{array}{rll}
(1-xwaw)aw&=&[1-xwaw](z+y)w=[1-xwaw]zw+yw\\
&=&[1-xw(z+y)w]zw+yw\\
&=&[1-xwzw]zw+yw=yw.
\end{array}$$

Obviously, we have $$\begin{array}{rll}
[aw-xw(aw)^2]xw(aw)^2&=&[awxw-xwaw(awxw)](aw)^2\\
&=&[awxw-zwxw](aw)^2=0.
\end{array}$$
Hence we derive that $$\lim\limits_{n\to \infty}||(aw)^n-xw(aw)^{n+1}||^{\frac{1}{n}}=0.$$
Therefore $a\in \mathcal{A}_w^{\tiny\textcircled{e}}$.
\end{proof}

\begin{cor} Let $a\in \mathcal{A}$. Then the following are equivalent:\end{cor}
\begin{enumerate}
\item [(1)] $a\in \mathcal{A}^{\tiny\textcircled{e}}$.
\vspace{-.5mm}
\item [(2)] There exists $x\in \mathcal{A}$ such that $$\begin{array}{c}
ax^2=x, (ax)^*=ax, (xa)^*=xa, \lim\limits_{n\to \infty}||a^n-xa^{n+1}||^{\frac{1}{n}}=0.
\end{array}$$
\end{enumerate}
\begin{proof} This is obvious by Theorem 3.1.\end{proof}

We are ready to prove:

\begin{thm} Let $a\in \mathcal{A}$. Then the following are equivalent:\end{thm}
\begin{enumerate}
\item [(1)] $a\in \mathcal{A}_w^{\tiny\textcircled{e}}$.
\vspace{-.5mm}
\item [(2)] There exists $x\in \mathcal{A}$ such that
$$a(wx)^2=x, (awxw)^*=xwaw, aw-xw(aw)^2\in \mathcal{A}^{qnil}.$$
\end{enumerate}
\begin{proof} $(1)\Rightarrow (2)$ By virtue of Theorem 3.1, there exist $z,y\in \mathcal{A}$ such that $a=z+y, z^*y=ywz=0, z\in \mathcal{A}_w^{\tiny{e}},
y\in \mathcal{A}^{qnil}.$ Set $x=z_w^{\tiny{e}}$. In view of Theorem 2.1, we can choose $x=z_w^{\#}$. Then
$$\begin{array}{rll}
a(wx)^2&=&(z+y)(wx)^2=z(wz_w^{\#})^2=x,\\
xwaw&=&xw(z+y)w=(xw)(zwxw)(zw+yw)\\
&=&(xw)(zwxw)^*(zw+yw)=(xw)(zwxw)^*(zw)\\
&=&xwzwxwzw=zwxw=(z+y)wxw=awxw,\\
(awxw)^*&=&(zwxw)^*=zwxw=xwaw,\\
aw-xw(aw)^2&=&(1-xwaw)aw=(1-zwxw)aw\\
&=&aw-zwxw(z+y)w=aw-zwxwzw\\
&=&(a-z)w=yw\in \mathcal{A}^{qnil},
\end{array}$$ as desired.

$(2)\Rightarrow (1)$ By hypotheses, we have $z\in \mathcal{A}$ such that $$a(wz)^2=z, (awzw)^*=zwaw, aw-zw(aw)^2\in \mathcal{A}^{qnil}.$$
Then $a(wz)^2=z, (zwaw)^*=awzw$. As in the proof of Theorem 2.6,
we see that $zwaw=awzw$. Set $x=awzwa$ and $y=a-awzwa.$
We claim that $x$ is a $w$-EP element.
Evidently, we verify that
$$\begin{array}{rll}
z(wx)^2&=&zwawzw(aw)^2zwa=(aw)^3(zw)^3a=awzwa=x,\\
xwzw&=&awzwawzw=(aw)^2(zw)^2=awzw\\
&=&zwaw=aw(zw)^2aw=zwawzwaw=zwxw,\\
(xwzw)^*&=&(awzw)^*=zwaw=awzw\\
&=&(aw)^2(zw)^2=(zwaw)zwaw=zwxw.
\end{array}$$ In light of Theorem 2.6, $x\in \mathcal{A}_w^{\tiny{e}}$.

By hypothesis, $(1-zwaw)aw\in \mathcal{A}^{qnil}$. By using Cline's formula (see~\cite[Theorem 2.1]{L}),
$aw(1-zwaw)\in \mathcal{A}^{qnil}$; hence,
$y=a-awzwa\in \mathcal{A}_w^{qnil}$. Furthermore, we verify that $$\begin{array}{rll}
x^*y&=&(awzwa)^*(1-awzw)a=a^*(awzw)^*(1-awzw)a\\
&=&a^*(awzw)(1-awzw)a=0,\\
ywx&=&(a-awzwa)wawzwa=awawzwa-awzw(awaw)zwa\\
&=&awawzwa-[a(wz)^2]w(aw)^2a\\
&=&(aw)^2zwa-zw(aw)^2a=0\\
\end{array}$$ This proof is completed.\end{proof}

\begin{cor} Let $a\in \mathcal{A}$. Then the following are equivalent:\end{cor}
\begin{enumerate}
\item [(1)] $a\in \mathcal{A}^{\tiny\textcircled{e}}$.
\vspace{-.5mm}
\item [(2)] There exists $x\in \mathcal{A}$ such that
$$ax^2=x, (ax)^*=xa, a-xa^2\in \mathcal{A}^{qnil}.$$
\end{enumerate}
\begin{proof} Straightforward by Theorem 3.3.\end{proof}

We come now to establish the polar-like property of generalized $w$-EP elements.

\begin{thm} Let $a\in \mathcal{A}$. Then the following are equivalent:\end{thm}
\begin{enumerate}
\item [(1)]{\it $a\in \mathcal{A}_w^{\tiny\textcircled{e}}$.}
\item [(2)]{\it There exists a projection $p\in \mathcal{A}$ such that
$$aw+p\in \mathcal{A}^{-1}, 1-p\in \mathcal{A}w, awp=paw\in \mathcal{A}^{qnil}.$$}
\end{enumerate}
\begin{proof} $(1)\Rightarrow (2)$ As in the proof of Theorem 3.1, there exists $x\in \mathcal{A}$ such that
$$a(wx)^2=x, (awxw)^*=xwaw, xwaw=awxw, aw-xw(aw)^2\in \mathcal{A}^{qnil}.$$
Let $p=1-awxw$. Then $p=p^2=p^*\in \mathcal{A}, 1-p\in \mathcal{A}w$ and $aw+p=aw+1-awxw$. It is easy to verify that
$$\begin{array}{rl}
&(aw+1-awxw)(xw+1-awxw)\\
=&(xw+1-awxw)(aw+1-awxw),\\
&(aw+1-awxw)(xw+1-awxw)\\
=&awxw+aw(1-awxw)+(1-awxw)xw+(1-awxw)\\
=&1+aw(1-awxw)\in \mathcal{A}^{-1}.
\end{array}$$ Therefore $aw+p\in \mathcal{A}^{-1}$. Morover, we have $awp=paw=aw-(aw)^2xw\in \mathcal{A}^{qnil}$, as
required.

$(2)\Rightarrow (1)$ By hypothesis, There exists a projection $p\in \mathcal{A}$ such that
$$aw+p\in \mathcal{A}^{-1}, 1-p\in \mathcal{A}w, awp=paw\in \mathcal{A}^{qnil}.$$ Set
$x=(1-p)(aw+p)^{-1}$. Then $x\in \mathcal{A}w$. Write $x=zw$ with $(1-p)z=z$. Then
 $$\begin{array}{rll}
 a(wz)^2&=&awzwz=awxz=aw(1-p)(aw+p)^{-1}z\\
 &=&(1-p)(aw+p)(aw+p)^{-1}z=(1-p)z=z,\\
 awzw&=&awx=aw(1-p)(aw+p)^{-1}\\
 &=&(1-p)(aw+p)(aw+p)^{-1}=1-p,\\
 zwaw&=&xaw=(1-p)(aw+p)^{-1}aw=(aw+p)^{-1}aw(1-p)\\
 &=&(aw+p)^{-1}(aw+p)(1-p)=1-p,\\
 (awzw)^*&=&zwaw.
 \end{array}$$ Moreover,
we verify that $$\begin{array}{rll}
aw-zw(aw)^2&=&aw-(1-p)(aw+p)^{-1}(aw)^2\\
&=&aw-(aw+p)^{-1}(aw+p)aw(1-p)\\
&=&aw-aw(1-p)=awp\in
\mathcal{A}^{qnil}.
\end{array}$$ According to Theorem 3.3, $a\in \mathcal{A}_w^{\tiny\textcircled{e}}$.\end{proof}

\begin{cor} Let $a\in \mathcal{A}$. Then the following are equivalent:\end{cor}
\begin{enumerate}
\item [(1)] $a\in \mathcal{A}_w^{\tiny\textcircled{e}}$.
\vspace{-.5mm}
\item [(2)] There exists $x\in \mathcal{A}$ such that $$(awxw)^*=awxw=xwaw, aw-xw(aw)^2\in \mathcal{A}^{qnil}.$$
\end{enumerate}
\begin{proof} $(1)\Rightarrow (2)$ This is proved as in Theorem 3.3.

$(2)\Rightarrow (1)$ By hypothesis, there exists some $x\in \mathcal{A}$ such that $$(awxw)^*=awxw=xwaw, aw-xw(aw)^2\in \mathcal{A}^{qnil}.$$
Set $s=xwawx$.
Then we check that $$\begin{array}{rll}
aw-(aw)^2sw&=&aw-(aw)^3(xw)^2\\
&=&(1+awxw)[aw-(aw)^2xw]\\
&\in& \mathcal{A}^{qnil},\\
sw-(sw)^2aw&=&xwawxw-(xwawxw)^2aw\\
&=&awxwxw-(aw)^3(xw)^4\\
&=&(1+awxw)(xw)^2[aw-(aw)^2xw]\\
&\in&\mathcal{A}^{qnil}.
\end{array}$$
Since $(aw)(xw)=(xw)(aw)$, we verify that $(aw)(sw)=(sw)(aw)$.
In view of ~\cite[Lemma 2.4]{CK}, we have $$awsw-(awsw)^2=[aw-(aw)^2sw]sw\in \mathcal{A}^{qnil}.$$ Let $$\begin{array}{c}
q=awsw-(awsw)^2, z=-\frac{1}{2}\sum\limits_{k=1}^{\infty}
\left(
\begin{array}{c}
\frac{1}{2}\\
k
\end{array}
\right)\big(4q(4q-1)^{-1}\big)^k,\\
e=awsw-(2awsw-1)z.
\end{array}$$ As in the proof of
~\cite[Lemma 15.1.2]{C1}, we have $e^2=e$ and $awzw-e\in \mathcal{A}^{qnil}$.
We easily check that
$$\begin{array}{rl}
&(aw+1-awsw)(sw+1-awsw)\\
=&1+[aw-(aw)^2sw](1-sw)+[sw-(sw)^2aw]\in \mathcal{A}^{-1}.
\end{array}$$
Hence, $$\begin{array}{lll}
aw+1-e&=&(aw+1-awsw)+(awsw-e)\in \mathcal{A}^{-1},\\
aw(1-e)&=&[aw-(aw)^2sw]+aw(awsw-e)\in\mathcal{A}^{qnil}.
\end{array}$$
Since $(awxw)^*=awxw$, we see that $(awsw)^*=awsw$. This implies that $q^*=q$, and then $z^*=z$.
Therefore $e^*=e$. Moreover, we see that $e\in \mathcal{A}w$. By virtue of
Theorem 3.5, $a\in \mathcal{A}_w^{\tiny\textcircled{e}}$.\end{proof}

\section{Representations of generalized weighted EP inverses}

The aim of this section is to establish representations of generalized weighted EP inverses by using other weighted generalized inverses. We come now to the demonstration for which this section has been developed.

\begin{thm} Let $a\in \mathcal{A}$. Then the following are equivalent:\end{thm}
\begin{enumerate}
\item [(1)] $a\in \mathcal{A}_w^{\tiny\textcircled{e}}$.
\vspace{-.5mm}
\item [(2)] $a\in \mathcal{A}^{d,w}$ and $a^{d,w}\in \mathcal{A}_w^{\tiny{e}}$.
\end{enumerate}
In this case, $$[a^{d,w}]_w^{\tiny{e}}=(aw)^2a_w^{\tiny\textcircled{e}}, a_w^{\tiny\textcircled{e}}=[a^{d,w}w]^2[a^{d,w}]_w^{\tiny{e}}.$$
\begin{proof} $(1)\Rightarrow (2)$ In view of Theorem 3.1, there exist $x,y\in \mathcal{A}$ such that $$a=x+y, x^*y=ywx=0, x\in \mathcal{A}_w^{\tiny{e}},~y\in \mathcal{A}_w^{qnil}.$$ Since $xw,yw\in \mathcal{A}^d$ and $(yw)(xw)=0$, it follows by ~\cite[Lemma 15.2.2]{C1} that
$aw\in \mathcal{A}^d$. Hence $a\in \mathcal{A}^{d,w}$ and $a^{d,w}=[(aw)^d]^2a$.
We verify that $$\begin{array}{rll}
[(aw)^2a_w^{\tiny\textcircled{e}}w][((aw)^d)^2aw]&=&(aw)^2a_w^{\tiny\textcircled{e}}w(aw)^2[(aw)^d]^3\\
&=&(aw)^3[(aw)^d]^3=(aw)^daw\\
&=&awa_w^{\#}=awa_w^{\tiny\textcircled{e}}w,\\
((aw)^d)^2a[w(aw)^2a_w^{\tiny\textcircled{e}}]^2&=&(aw)^d(aw)^3a_w^{\tiny\textcircled{e}}]\\
&=&(aw)^2a_w^{\tiny\textcircled{e}},\\
(aw)^2a_w^{\tiny\textcircled{e}}[w((aw)^d)^2a]^2&=&(aw)^2a_w^{\tiny\textcircled{e}}w((aw)^d)^3a\\
&=&(aw)^2a_w^{\tiny\textcircled{e}}w(aw)^2[(aw)^d]^5a\\
&=&(aw)^3[(aw)^d]^5a=((aw)^d)^2a.
\end{array}$$
Thus $((aw)^d)^2a\in \mathcal{A}_w^{\tiny{e}}$ and
$$[a^{d,w}]_w^{\tiny{e}}=(aw)^2a_w^{\tiny\textcircled{e}}.$$

$(2)\Rightarrow (1)$ Set $x=[a^{d,w}w]^2[a^{d,w}]_w^{\tiny{e}}$. Then we check that
$$\begin{array}{rll}
(xw)\big(a^{d,w}w[a^{d,w}]_w^{\tiny{e}}waw\big)&=&[a^{d,w}w]^2[a^{d,w}]_w^{\tiny{e}}wa^{d,w}w[a^{d,w}]_w^{\tiny{e}}waw\\
&=&[a^{d,w}w]^2[a^{d,w}]_w^{\tiny{e}}waw\\
&=&[a^{d,w}]_w^{\tiny{e}}w[a^{d,w}w]^2aw\\
&=&[a^{d,w}]_w^{\tiny{e}}wa^{d,w}w.
\end{array}$$ Hence, $$[(xw)\big(a^{d,w}w[a^{d,w}]_w^{\tiny{e}}waw\big)]^*=(xw)\big(a^{d,w}w[a^{d,w}]_w^{\tiny{e}}waw\big).$$
Moreover, we see that
$$\begin{array}{rll}
[a^{d,w}w[a^{d,w}]_w^{\tiny{e}}wa](wx)^2&=&[a^{d,w}w[a^{d,w}]_w^{\tiny{e}}wa]w[a^{d,w}w]^2[a^{d,w}]_w^{\tiny{e}}wx\\
&=&a^{d,w}w[a^{d,w}]_w^{\tiny{e}}wa^{d,w}w[a^{d,w}]_w^{\tiny{e}}wx\\
&=&[a^{d,w}w]^2[([a^{d,w}]_w^{\tiny{e}}w)^2a^{d,w}]w[a^{d,w}w][a^{d,w}]_w^{\tiny{e}}\\
&=&[a^{d,w}w]^2[a^{d,w}]_w^{\tiny{e}}wa^{d,w}w[a^{d,w}]_w^{\tiny{e}}\\
&=&[a^{d,w}w]^2[a^{d,w}]_w^{\tiny{e}}=x,\\
x[wa^{d,w}w[a^{d,w}]_w^{\tiny{e}}wa]^2&=&[a^{d,w}w]^2[a^{d,w}]_w^{\tiny{e}}wawa^{d,w}w[a^{d,w}]_w^{\tiny{e}}wa\\
&=&[a^{d,w}w][a^{d,w}]_w^{\tiny{e}}wa^{d,w}wawa^{d,w}w[a^{d,w}]_w^{\tiny{e}}wa\\
&=&[a^{d,w}w][a^{d,w}]_w^{\tiny{e}}wa^{d,w}w[a^{d,w}]_w^{\tiny{e}}wa\\
&=&[a^{d,w}w][a^{d,w}]_w^{\tiny{e}}wa.
\end{array}$$ Then $$\begin{array}{c}
a^{d,w}w[a^{d,w}]_w^{\tiny{e}}wa\in \mathcal{A}_w^{\tiny{e}},\\
\big(a^{d,w}w(a^{d,w})_w^{\tiny{e}}wa\big)_w^{\tiny{e}}=[a^{d,w}w]^2\big(a^{d,w}\big)_w^{\tiny{e}}.
\end{array}$$ Write $a=a_1+a_2$, where $a_1=a^{d,w}w[a^{d,w}]_w^{\tiny{e}}wa$ and $a_2=a-a^{d,w}w[a^{d,w}]_w^{\tiny{e}}wa$. It is easy to verify that
$$\begin{array}{rl}
&a_2wa_1\\
=&[a-a^{d,w}w[a^{d,w}]_w^{\tiny{e}}wa]wa^{d,w}w[a^{d,w}]_w^{\tiny{e}}wa\\
=&[awa^{d,w}w[a^{d,w}]_w^{\tiny{e}}-a^{d,w}[((aw)^da)^2]_w^{\tiny{e}}w[(aw)^d]^2a]waw[a^{d,w}]_w^{\tiny{e}}]wa\\
=&[awa^{d,w}w[a^{d,w}]_w^{\tiny{e}}-a^{d,w}waw[a^{d,w}]_w^{\tiny{e}}]wa\\
=&0,\\
&a_1^*a_2\\
=&[a^{d,w}w[a^{d,w}]_w^{\tiny{e}}wa]^*[a-a^{d,w}w[a^{d,w}]_w^{\tiny{e}}wa]\\
=&a^*[a^{d,w}w[a^{d,w}]_w^{\tiny{e}}w][1-a^{d,w}w[a^{d,w}]_w^{\tiny{e}}w]a\\
=&0.
\end{array}$$
Moreover, we check that
$$\begin{array}{rll}
aw-a^{d,w}w[a^{d,w}]_w^{\tiny{e}}waw&=&aw-[a^{d,w}]_w^{\tiny{e}}wa^{d,w}waw\\
&=&aw-[a^{d,w}]_w^{\tiny{e}}w(aw)^2[(aw)^d]^2\\
&=&aw-aw[(aw)^d]^2\\
&\in& \mathcal{A}^{qnil}.
\end{array}$$
Thus, $a_2=a-a^{d,w}w[a^{d,w}]_w^{\tiny{e}}wa\in \mathcal{A}_w^{qnil}$. Therefore $a=a_1+a_2$ is the generalized $w$-core decomposition of $a$. Then
$a_w^{\tiny\textcircled{d}}=(a_1)_w^{\tiny{e}}=[a^{d,w}]^2[a^{d,w}]_w^{\tiny{e}}.$\end{proof}

\begin{cor} Let $a\in \mathcal{A}$. Then the following are equivalent:\end{cor}
\begin{enumerate}
\item [(1)] $a\in \mathcal{A}^{\tiny\textcircled{e}}$.
\vspace{-.5mm}
\item [(2)] $a\in \mathcal{A}^{d}$ and $a^{d}\in \mathcal{A}^{\tiny{EP}}$.
\end{enumerate}
In this case, $$a^{\tiny\textcircled{e}}=(a^{d})^2(a^{d})^{\tiny{d}}.$$
\begin{proof} This is obvious by Theorem 4.1.\end{proof}

\begin{thm} Let $a,w\in \mathcal{A}$. Then the following are equivalent:\end{thm}
\begin{enumerate}
\item [(1)] $a\in \mathcal{A}_w^{\tiny{e}}$.
\vspace{-.5mm}
\item [(2)] $a\in \mathcal{A}^{d,w}$ and there exists $x\in \mathcal{A}$ such that
$$xw(aw)^dx=x, x\mathcal{A}=(aw)^d\mathcal{A}, \mathcal{A}x^*=\mathcal{A}(aw)^d.$$
\item [(3)] $a\in \mathcal{A}^{d,w}$ and there exists $x\in \mathcal{A}$ such that
$$xw(aw)^dx=x, \ell(x)=\ell(aw)^d, r(x^*)=r(aw)^d.$$
\end{enumerate}
\begin{proof} $(1)\Rightarrow (2)$ Let $x=a_w^{\tiny\textcircled{e}}wawa$. Then we check that
$$\begin{array}{rll}
xw(aw)^dx&=&a_w^{\tiny\textcircled{e}}wawaw(aw)^da_w^{\tiny\textcircled{e}}wawa\\
&=&aw(aw)^da_w^{\tiny\textcircled{e}}wawa_w^{\tiny\textcircled{e}}wawa\\
&=&aw(aw)^da_w^{\tiny\textcircled{e}}wawa\\
&=&x.
\end{array}$$  Hence $x=xw(aw)^dx=a_w^{\tiny\textcircled{e}}w(aw)^2(aw)^dx=(aw)^da_w^{\tiny\textcircled{e}}w(aw)^2x.$

On the other hand, $$(aw)^d=(aw)^n[(aw)^d]^{n+1}=[(aw)^n-a_w^{\tiny\textcircled{e}}waw(aw)^n][(aw)^d]^{n+1}+xw(aw)^d.$$
This implies that $(aw)^d=xw(aw)(aw)^d$. Thus, we deduce that
$x\mathcal{A}=(aw)^d\mathcal{A}$.

On the other hand, we have $$x^*=[a_w^{\tiny\textcircled{e}}wawa]^*=a^*[a_w^{\tiny\textcircled{e}}waw]=a^*[a_w^{\tiny\textcircled{e}}w(aw)^2](aw)^d.$$
Then $\mathcal{A}x^*\subseteq \mathcal{A}(aw)^d$.
Observing that $$\begin{array}{rll}
(aw)^d&=&(aw)^da_w^{\tiny\textcircled{e}}waw\\
&=&(aw)^d[a_w^{\tiny\textcircled{e}}waw][a_w^{\tiny\textcircled{e}}waw]\\
&=&(aw)^d[awa_w^{\tiny\textcircled{e}}w]^*[a_w^{\tiny\textcircled{e}}waw]\\
&=&(aw)^d[wa_w^{\tiny\textcircled{e}}w]^*a^*[a_w^{\tiny\textcircled{e}}waw]^*\\
&=&(aw)^d[wa_w^{\tiny\textcircled{e}}w]^*[a_w^{\tiny\textcircled{e}}wawa]^*\\
&=&(aw)^d[wa_w^{\tiny\textcircled{e}}w]^*x^*,
\end{array}$$ we see that $\mathcal{A}(aw)^d\subseteq \mathcal{A}x^*$. Therefore $\mathcal{A}x^*=\mathcal{A}(aw)^d$, as required.

$(2)\Rightarrow (3)$ If $rx=0$ for $r\in \mathcal{A}$, then $r(aw)^d=0$. If $r(aw)^d=0$ for $r\in \mathcal{A}$, then $rx=0$.
Hence $\ell(x)=\ell(aw)^d$. Likewise, we have $r(x^*)=r(aw)^d,$ as required.

$(3)\Rightarrow (1)$ By hypothesis, there exists $x\in \mathcal{A}$ such that
$$xwa^{d,w}wx=x, \ell(x)=\ell(a^{d,w}w), r(x^*)=r(a^{d,w}w).$$
In light of Theorem 2.10, $a^{d,w}\in \mathcal{A}_w^{\tiny\textcircled{\#}}$. Therefore $a\in \mathcal{A}_w^{\tiny{e}}$ by Theorem 4.1.\end{proof}

\begin{cor} Let $a,w\in \mathcal{A}$. Then the following are equivalent:\end{cor}
\begin{enumerate}
\item [(1)] $a\in \mathcal{A}^{\tiny{e}}$.
\vspace{-.5mm}
\item [(2)] $a\in \mathcal{A}^{d}$ and there exists $x\in \mathcal{A}$ such that
$$xa^dx=x, x\mathcal{A}=a^d\mathcal{A}, \mathcal{A}x^*=\mathcal{A}a^d.$$
\item [(3)] $a\in \mathcal{A}^{d}$ and there exists $x\in \mathcal{A}$ such that
$$xa^dx=x, \ell(x)=\ell(a^d), r(x^*)=r(a^d).$$
\end{enumerate}
\begin{proof} Straightforward by Theorem 4.3.\end{proof}

We now proceed to examine the weighted EP property of the generalized weighted EP inverse.

\begin{thm} Let $a\in \mathcal{A}_w^{\tiny\textcircled{e}}$. Then $a_w^{\tiny\textcircled{e}}\in \mathcal{A}_w^{\tiny{e}}$.
In this case, $$(a_w^{\tiny\textcircled{e}})_w^{\tiny{e}}=(aw)^2a_w^{\tiny\textcircled{e}}.$$\end{thm}
\begin{proof} In view of Theorem 4.1, $a\in \mathcal{A}^{d,w}$. Let $x=(aw)^2a_w^{\tiny\textcircled{e}}$. Then we check that
$$\begin{array}{rll}
a_w^{\tiny\textcircled{e}}(wx)^2&=&[a_w^{\tiny\textcircled{e}}w](aw)^2a_w^{\tiny\textcircled{e}}wx\\
&=&(aw)[a_w^{\tiny\textcircled{e}}wawa_w^{\tiny\textcircled{e}}]wx\\
&=&(aw)a_w^{\tiny\textcircled{e}}w(aw)^2a_w^{\tiny\textcircled{e}}\\
&=&(aw)^2[a_w^{\tiny\textcircled{e}}wawa_w^{\tiny\textcircled{e}}]\\
&=&(aw)^2a_w^{\tiny\textcircled{e}}=x,\\
x(wa_w^{\tiny\textcircled{e}})^2&=&aw[a_w^{\tiny\textcircled{e}}wawa_w^{\tiny\textcircled{e}}]wa_w^{\tiny\textcircled{e}}\\
&=&a_w^{\tiny\textcircled{e}}wawa_w^{\tiny\textcircled{e}}\\
&=&a_w^{\tiny\textcircled{e}},\\
xwa_w^{\tiny\textcircled{e}}w&=&(aw)^2a_w^{\tiny\textcircled{e}}wa_w^{\tiny\textcircled{e}}w\\
&=&aw[a_w^{\tiny\textcircled{e}}wawa_w^{\tiny\textcircled{e}}]w\\
&=&awa_w^{\tiny\textcircled{e}}w,\\
(xwa_w^{\tiny\textcircled{e}}w)^*&=&xwa_w^{\tiny\textcircled{e}}w.
\end{array}$$

By virtue of Theorem 2.1, $a_w^{\tiny\textcircled{e}}\in \mathcal{A}_w^{\tiny{e}}$. Moreover, we have
$(a_w^{\tiny\textcircled{e}})_w^{\tiny{e}}=(aw)^2a_w^{\tiny\textcircled{e}}.$\end{proof}

\begin{cor} Let $a\in \mathcal{A}_w^{\tiny\textcircled{e}}$. Then
$[(a_w^{\tiny\textcircled{e}})_w^{\tiny\textcircled{e}}]_w^{\tiny\textcircled{e}}=a_w^{\tiny\textcircled{e}}.$
\end{cor}
\begin{proof} By using Theorem 4.5, we have $$(a_w^{\tiny\textcircled{e}})_w^{\tiny\textcircled{e}}=(a_w^{\tiny\textcircled{e}})_w^{\tiny{e}}=(aw)^2a_w^{\tiny\textcircled{e}}.$$
Moreover, we derive that $$\begin{array}{rll}
\big((a_w^{\tiny\textcircled{e}})_w^{\tiny\textcircled{e}}\big)_w^{\tiny\textcircled{e}}&=&
[a_w^{\tiny\textcircled{e}}w]^2(a_w^{\tiny\textcircled{e}})_w^{\tiny\textcircled{e}}\\
&=&[a_w^{\tiny\textcircled{e}}w]^2(aw)^2a_w^{\tiny\textcircled{e}}\\
&=&[a_w^{\tiny\textcircled{e}}w]aw[a_w^{\tiny\textcircled{e}}wawa_w^{\tiny\textcircled{e}}]\\
&=&a_w^{\tiny\textcircled{e}}wawa_w^{\tiny\textcircled{e}}\\
&=&a_w^{\tiny\textcircled{d}}.
\end{array}$$ This completes the proof.\end{proof}

\section{Weighted *-DMP elements}

In this section, we are concerned with weighted *-DMP elements in a Banach *-algebra. Many properties of *-DMP elements are thereby extended to the wider cases.

\begin{thm} Let $a\in \mathcal{A}$. Then the following are equivalent:\end{thm}
\begin{enumerate}
\item [(1)] $a\in \mathcal{A}_w^{\tiny\textcircled{E}}$.
\vspace{-.5mm}
\item [(2)] There exist $x,y\in \mathcal{A}$ such that $$a=x+y, x^*y=ywx=0, x\in \mathcal{A}_w^{\tiny{e}},~y\in \mathcal{A}_w^{nil}.$$
\end{enumerate}
\begin{proof} $(1)\Rightarrow (3)$ By hypothesis, we can find $x\in \mathcal{A}$ such that
$$a(wx)^2=x, (awxw)^*=awxw, (xwaw)^*=xwaw, (aw)^n=xw(aw)^{n+1}.$$ Let $z=awxwa$ and $y=a-awxwa$. As in the proof of Theorem 3.1, we verify that
$$z^*y=0, ywz=0, z_w^{\tiny{e}}=x.$$ We observe that
$$\begin{array}{rll}
yw&=&aw-awxwaw,\\
\big(aw-xw(aw)^2\big)^2&=&[aw-xw(aw)^2]aw,\\
\big(aw-xw(aw)^2\big)^n&=&(aw)^n-(xw)(aw)^{n+1}=0.
\end{array}$$ Then $aw-xw(aw)^2$ is nilpotent, and so $yw=aw-awxwaw\in \mathcal{A}^{nil}$, as desired.

$(2)\Rightarrow (1)$ By hypothesis, there exist $z,y\in \mathcal{A}$ such that $a=z+y, z^*y=ywz=0, z\in \mathcal{A}_w^{\tiny{e}},
y\in \mathcal{A}^{nil}.$ Set $x=z_w^{\tiny{e}}$. In light of Lemma 2.9, we have
$xwzwx=x.$ Moreover, we check that
$$a(wx)^2=x, (awxw)^*=awxw, (xwaw)^*=xwaw.$$ Write $(yw)^n=0$ for some $n\in {\Bbb N}$. Then
$$\begin{array}{rll}
(1-xwaw)yw&=&yw-xw(z+y)wyw=yw-xwzwyw+xwywyw\\
&=&yw-(zwxw)^*yw+x(wz)^2wywyw\\
&=&yw-(wxw)^*(z^*y)w+xwzwzwywyw\\
&=&yw-(wxw)^*(z^*y)w+zw(xwzw)ywyw\\
&=&yw-(wxw)^*(z^*y)w+zw(zwxw)ywyw\\
&=&yw-(wxw)^*(z^*y)w+zw(zwxw)^*ywyw\\
&=&yw-(wxw)^*(z^*y)w+zw(wxw)^*(z^*y)wyw\\
&=&yw.
\end{array}$$ Hence, we prove that
$$\begin{array}{rl}
&(aw)^n-xw(aw)^{n+1}\\
=&[1-xwaw](aw)^n=[(1-xwaw)aw](aw)^{n-1}\\
=&[(1-xwaw)(z+y)w](aw)^{n-1}=[(1-xwaw)zw+yw](aw)^{n-1}\\
=&[(1-xw(z+y)w)zw+yw](aw)^{n-1}\\
=&[(1-xwzw)zw+yw](aw)^{n-1}=yw(aw)^{n-1}=(yw)^n=0.
\end{array}$$
THerefore $(aw)^n=xw(aw)^{n+1}$, as asserted.\end{proof}

\begin{cor} Let $a,b\in \mathcal{A}_w^{\tiny\textcircled{E}}$. If $awb=bwa=a^*b=0$, then $a+b\in \mathcal{A}_w^{\tiny\textcircled{E}}$. In this case,
$$(a+b)_w^{\tiny\textcircled{E}}=a_w^{\tiny\textcircled{E}}+b_w^{\tiny\textcircled{E}}.$$\end{cor}
\begin{proof} In view of Theorem 5.1, we have weighted generalized EP-decompositions:
$$\begin{array}{c}
a=x+y, x^*y=ywx=0, x\in \mathcal{A}_w^{\tiny\textcircled{\#}}, y\in \mathcal{A}_w^{nil};\\
b=s+t, s^*t=tws=0, s\in \mathcal{A}_w^{\tiny\textcircled{\#}}, t\in \mathcal{A}_w^{nil}.
\end{array}$$ Explicitly, we have $x=awa_w{\tiny\textcircled{D}}wa$ and $s=bwb_w^{\tiny\textcircled{E}}wb$.
Then $a+b=(x+s)+(y+t)$. We directly check that
$$\begin{array}{rll}
(x+s)w(x_w^{\tiny\textcircled{\#}}+s_w^{\tiny\textcircled{\#}})&=&xwx_w^{\tiny\textcircled{\#}}+sws_w^{\tiny\textcircled{\#}}\\
&=&x_w^{\tiny\textcircled{\#}}wx+s_w^{\tiny\textcircled{\#}}ws\\
&=&(x_w^{\tiny\textcircled{\#}}+s_w^{\tiny\textcircled{\#}})w(x+s),\\
\big((x+s)w\big)^2[x_w^{\tiny\textcircled{\#}}+s_w^{\tiny\textcircled{\#}}]&=&(xw)^2x_w^{\tiny\textcircled{\#}}+(sw)^2s_w^{\tiny\textcircled{\#}}\\
&=&x+s,\\
(x+s)w(x_w^{\tiny\textcircled{\#}}+s_w^{\tiny\textcircled{\#}})w&=&xwx_w^{\tiny\textcircled{\#}}w+sws_w^{\tiny\textcircled{\#}})w,\\
\big((x+s)w(x_w^{\tiny\textcircled{\#}}+s_w^{\tiny\textcircled{\#}})w\big)^*&=&(x+s)w(x_w^{\tiny\textcircled{\#}}+s_w^{\tiny\textcircled{\#}})w.
\end{array}$$
Then $x+s\in \mathcal{A}_w^{\tiny\textcircled{\#}}$ and $(x+s)_w^{\tiny\textcircled{\#}}=x_w^{\tiny\textcircled{\#}}+s_w^{\tiny\textcircled{\#}}.$
Since $ywt=(a-awa_w{\tiny\textcircled{D}}wa)w(b-bwb_w^{\tiny\textcircled{E}}wb)=0$, we have $y+t\in \mathcal{A}_w^{nil}$(see~\cite[Lemma 15.2.2]{C1}).

Obviously, we check that $$\begin{array}{rll}
(x+s)^*(y+t)&=&x^*y+x^*t+s^*y+s^*t=x^*t+s^*y\\
&=&(wa_w{\tiny\textcircled{D}}wa)^*(a^*b)(1-wb_w^{\tiny\textcircled{E}}wb)\\
&+&(wb_w^{\tiny\textcircled{E}}wb)^*(b^*a)(1-wa_w{\tiny\textcircled{D}}wa)=0,\\
(y+t)w(x+s)&=&ywx+yws+twx+tws=yws+twx\\
&=&(a-awa_w{\tiny\textcircled{D}}wa)wbwb_w^{\tiny\textcircled{E}}wb\\
&+&(b-bwb_w^{\tiny\textcircled{E}}wb)wawa_w{\tiny\textcircled{D}}wa=0.
\end{array}$$ By using Theorem 5.1, $(a+b)_w^{\tiny\textcircled{E}}=(x+s)_w{\tiny\textcircled{\#}}=
x_w^{\tiny\textcircled{\#}}+s_w^{\tiny\textcircled{\#}}=a_w^{\tiny\textcircled{E}}+b_w^{\tiny\textcircled{E}},$ as asserted.\end{proof}

We come now to derive a new characterizations of *-DMP elements.

\begin{cor} Let $a\in \mathcal{A}$. Then the following are equivalent:\end{cor}
\begin{enumerate}
\item [(1)] $a\in \mathcal{A}$ is *-DMP.
\vspace{-.5mm}
\item [(2)] There exist $x,y\in \mathcal{A}$ such that $$a=x+y, x^*y=yx=0, x\in \mathcal{A}^{\tiny{EP}},~y\in \mathcal{A}_w^{nil}.$$
\end{enumerate}
\begin{proof} This is obvious by Theorem 5.1.\end{proof}

\begin{lem} Let $a\in \mathcal{A}$. Then $a\in \mathcal{A}_w^{\tiny\textcircled{E}}$ if and only if $a\in \mathcal{A}_w^{\tiny\textcircled{e}}\bigcap \mathcal{A}^{D,w}$.\end{lem}
\begin{proof} $\Longrightarrow $ In view of Theorem 5.1, there exist $x,y\in \mathcal{A}$ such that $$a=x+y, x^*y=ywx=0, x\in \mathcal{A}_w^{\tiny{e}},~y\in \mathcal{A}_w^{nil}.$$ Hence, $aw=xw+yw$.
Since $x\in \mathcal{A}_w^{\tiny{e}}$, it follows by Corollary 2.5 that $x\in \mathcal{A}_w^{\#}$. Then $xw\in \mathcal{A}^{D}$.
On the other hand, $yw\in \mathcal{A}^{\#}\subseteq \mathcal{A}^D$. Since $(yw)(xw)=0$, $aw\in \mathcal{A}^D$ (see~\cite[Lemma 15.2.2]{C1}).
Therefore $a\in \mathcal{A}^{D,w}$, as desired.

$\Longleftarrow $ Let $x=a_w^{\tiny\textcircled{e}}$. Then
$$\begin{array}{c}
a(wx)^2=x, (awxw)^*=awxw, (xwaw)^*=xwaw,\\
\lim\limits_{n\to \infty}||(aw)^n-xw(aw)^{n+1}||^{\frac{1}{n}}=0.
\end{array}$$ Let $z=awxwa$ and $y=a-awxwa$. As in the proof of Theorem 3.1, we verify that
$$a=z+y, z^*y=ywz=0, z\in \mathcal{A}_w^{\tiny{e}},~y\in \mathcal{A}_w^{qnil}.$$
Since $a\in \mathcal{A}^{D,w}$, we see that $(aw)^m=(aw)^D(aw)^{m+1}$ for some $m\in {\Bbb N}$. In view of Theorem 4.1,
$$x=a_w^{\tiny\textcircled{e}}=[a^{D,w}w]^2[a^{D,w}]_w^{\tiny{e}}.$$
Hence, $yw=aw-awxwaw=(1-xwaw)aw$, and so $$\begin{array}{rl}
&(yw)^m\\
=&(1-xwaw)(aw)^m=(aw)^m-(xw)(aw)^{m+1}\\
=&(aw)^m-[a^{D,w}w]^2[a^{D,w}]_w^{\tiny{e}}w(aw)^{m+1}\\
=&(aw)^m-[(aw)^D]^2[((aw)^D)^2a]_w^{\tiny{e}}w(aw)^{m+1}\\
=&(aw)^m-[((aw)^D)^2a]_w^{\tiny{e}}w[(aw)^D]^2(aw)^{m+1}\\
=&\big(((aw)^D)^2a-[((aw)^D)^2a]_w^{\tiny{e}}[w((aw)^D)^2a]^2\big)w(aw)^{m+1}\\
=&0.
\end{array}$$
Thus $yw\in \mathcal{A}^{nil}$, and therefore $a\in \mathcal{A}_w^{\tiny\textcircled{E}}$ by Theorem 5.1.\end{proof}

\begin{thm} Let $a\in \mathcal{A}$. Then the following are equivalent:\end{thm}
\begin{enumerate}
\item [(1)] $a\in \mathcal{A}_w^{\tiny\textcircled{E}}$.
\vspace{-.5mm}
\item [(2)] There exists some $x\in \mathcal{A}$ such that $$a(wx)^2=x, (awxw)^*=xwaw, (aw)^n=xw(aw)^{n+1}$$
for some $n\in {\Bbb N}$.
\vspace{-.5mm}
\item [(3)] There exists a projection $p\in \mathcal{A}$ such that
$$aw+p\in \mathcal{A}^{-1}, 1-p\in \mathcal{A}w, awp=paw\in \mathcal{A}^{nil}.$$.
\vspace{-.5mm}
\item [(4)] $a\in \mathcal{A}^{D,w}$ and $a^{D,w}\in \mathcal{A}_w^{\tiny{e}}$.
\end{enumerate}
\begin{proof} $(1)\Rightarrow (2)$ In view of Theorem 3.3, there exists $x\in \mathcal{A}$ such that
$$a(wx)^2=x, (awxw)^*=xwaw, aw-xw(aw)^2\in \mathcal{A}^{qnil}$$ for some $n\in {\Bbb N}$.
As in the proof of Theorem 3.3, $aw-awa_w^{\tiny\textcircled{e}}waw\in \mathcal{A}^{nil}$.
Therefore $aw[1-awa_w^{\tiny\textcircled{e}}w]\in \mathcal{A}^{nil}$. This implies that $aw-xw(aw)^2\in \mathcal{A}^{nil}$.
Hence, $(aw)^n=xw(aw)^{n+1}$ for some $n\in {\Bbb N}$.

$(2)\Rightarrow (1)$ As in the proof of Theorem 3.3, $awxw=xwaw$. Then
$(awxw)^*=awxw$ and $(xwaw)^*=xwaw$, as required.

$(1)\Rightarrow (3)$ By virtue of Lemma 5.4, $a\in \mathcal{A}_w^{\tiny\textcircled{e}}\bigcap \mathcal{A}^{D,w}$. According to Theorem 3.5,
we can find a projection $p\in \mathcal{A}$ such that
$$aw+p\in \mathcal{A}^{-1}, 1-p\in \mathcal{A}w, awp=paw\in \mathcal{A}^{qnil}.$$ $awp=aw[1-awa_w^{\tiny\textcircled{e}}w]$.
By the argument above, we see tht  $aw-awa_w^{\tiny\textcircled{e}}waw\in \mathcal{A}^{nil}$.  Therefore $awp\in \mathcal{A}^{nil},$ as required.

$(3)\Rightarrow (1)$ According to Theorem 3.5, $a\in \mathcal{A}_w^{\tiny\textcircled{e}}$. As in the proof of ~\cite[Theorem 4.2]{KP}, we see that
$aw$ has Drazin inverse, and then $a\in \mathcal{A}^{D,w}$. Hence $a\in \mathcal{A}_w^{\tiny\textcircled{E}}$ by Lemma 5.4.

$(1)\Rightarrow (4)$ By the argument above, $a\in \mathcal{A}^{D,w}$ and $a^{D,w}=a^{d,w}$. By using Theorem
$a^{D,w}\in \mathcal{A}_w^{\tiny{e}}$.

$(4)\Rightarrow (1)$ In view of Theorem 4.1, $a\in \mathcal{A}_w^{\tiny\textcircled{e}}$. Therefore we complete the proof by Lemma 5.4.\end{proof}

\begin{cor} Let $a\in \mathcal{A}$. Then the following are equivalent:\end{cor}
\begin{enumerate}
\item [(1)] $a\in \mathcal{A}$ is *-DMP.
\vspace{-.5mm}
\item [(2)] There exists $x\in \mathcal{A}$ such that
$$ax^2=x, (ax)^*=xa, a^n=xa^{n+1}$$ for some $n\in {\Bbb N}$.
\vspace{-.5mm}
\item [(3)] There exists a projection $p\in \mathcal{A}$ such that
$$a+p\in \mathcal{A}^{-1}, ap=pw\in \mathcal{A}^{nil}.$$.
\item [(4)] $a\in \mathcal{A}^{D}$ and $a^{D}\in \mathcal{A}^{\tiny{EP}}$.
\end{enumerate}
\begin{proof} This is obvious by choosing $w=1$ in Theorem 5.5.\end{proof}

It is a well-established fact that an element $a$ in $\mathcal{A}$ is *-DMP if and only if there exists a $k\in {\Bbb N}$
such that $a^k$ is EP. We broaden the scope of this finding to a more general setting.

\begin{thm} Let $a\in \mathcal{A}$. Then the following are equivalent:\end{thm}
\begin{enumerate}
\item [(1)] $a\in \mathcal{A}_w^{\tiny\textcircled{E}}.$
\vspace{-.5mm}
\item [(2)] There exists $m\in {\Bbb N}$ such that $a(wa)^{k-1}\in \mathcal{A}_w^{\tiny{e}}$ for any $k\geq m$.
\vspace{-.5mm}
\item [(3)] $a(wa)^{k-1}\in \mathcal{A}_w^{\tiny{e}}$ for some $k\in {\Bbb N}$.
\end{enumerate}
\begin{proof} $(1)\Rightarrow (2)$ By hypothesis, there exist $z\in \mathcal{A}$ and $m\geq 2$ such that
$$a(wz)^2=z, (awzw)^*=zwaw, (aw)^{m-1}=(zw)(aw)^{m}.$$
Then we easily see that $awzw=zwaw$. Let $x=(zw)^{k-1}z (k\geq m)$. Then $(aw)^{k-1}=(zw)(aw)^{k}$, and so

$$\begin{array}{rll}
[a(wa)^{k-1}]wxw&=&[a(wa)^{k-1}]w(zw)^{k-1}zw\\
&=&(aw)^k(zw)^{k-1}zw=(aw)^2zwzw\\
&=&[zw(aw)^2]zw=(aw)(zw)=(zw)(aw)\\
&=&[a(wz)^2]waw=aw(zw)^2aw\\
&=&(zw)^2(aw)aw=(zw)^k(aw)^{k-1}aw\\
&=&(zw)^{k-1}zw[a(wa)^{k-1}]w\\
&=&xw[a(wa)^{k-1}]w,\\
\big((a(wa)^{k-1})wxw\big)^*&=&(awzw)^*=awzw=(a(wa)^{k-1})wxw,\\
(zw)^{k-1}z\big(wa(wa)^{k-1}\big)^2&=&(zw)^{k-1}z(wa(wa)^{k-1}\big)(wa(wa)^{k-1}\big)\\
&=&(zw)^{k}(aw)^ka(wa)^{k-1}=(zw)(aw)a(wa)^{k-1}\\
&=&(zw)(aw)^{k}a=(aw)^{k-1}a=a(wa)^{k-1}.
\end{array}$$

Therefore $a(wa)^{k-1}\in \mathcal{A}_w^{\tiny{e}}$ by Theorem 2.6.

$(2)\Rightarrow (3)$ This is trivial.

$(3)\Rightarrow (1)$ Lt $x=a(wa)^{k-2}w[a(wa)^{k-1}]_w^{\tiny{e}}$. Then we check that

$$\begin{array}{rll}
awx&=&awa(wa)^{k-2}w[a(wa)^{k-1}]_w^{\tiny{e}}=(aw)^k[a(wa)^{k-1}]_w^{\tiny{e}},\\
awxw&=&a(wa)^{k-1}w[a(wa)^{k-1}]_w^{\tiny{e}}w\\
&=&(aw)^k[a(wa)^{k-1}]_w^{\tiny{e}}w=a(wa)^{k-2}waw[a(wa)^{k-1}]_w^{\tiny{e}}w\\
&=&a(wa)^{k-2}w[a(wa)^{k-1}]_w^{\tiny{e}}waw=xwaw,\\
(awxw)^*&=&xwaw,\\
a(wx)^2&=&(aw)^k[a(wa)^{k-1}]_w^{\tiny{e}}wa(wa)^{k-2}w[a(wa)^{k-1}]_w^{\tiny{e}}\\
&=&(aw)^{k-1}[a(wa)^{k-1}]_w^{\tiny{e}}w[a(wa)^{k-1}]w[a(wa)^{k-1}]_w^{\tiny{e}}\\
&=&(aw)^{k-1}[a(wa)^{k-1}]_w^{\tiny{e}}=x,\\
xw(aw)^{k+1}&=&(awxw)(aw)^k=(aw)^k[a(wa)^{k-1}]_w^{\tiny{e}}w(aw)^k\\
&=&[a(wa)^{k-1}]_w^{\tiny{e}}w(aw)^k(aw)^k\\
&=&[a(wa)^{k-1}]_w^{\tiny{e}}[wa(wa)^{k-1}]^2w\\
&=&a(wa)^{k-1}w=(aw)^k.
\end{array}$$ Hence $a\in \mathcal{A}_w^{\tiny\textcircled{E}}$, as asserted.\end{proof}

\begin{cor} Let $a\in \mathcal{A}$. Then the following are equivalent:\end{cor}
\begin{enumerate}
\item [(1)] $a\in \mathcal{A}_w^{\tiny{E}}.$
\vspace{-.5mm}
\item [(2)] $(aw)^{k-1}a\in \mathcal{A}_w^{\#}$ and $(aw)^{k-1}a\mathcal{A}=\big((aw)^k\big)^*\mathcal{A}$ for some $k\in {\Bbb N}$.
\end{enumerate}
\begin{proof} Since $[(aw)^{k-1}a]w=(aw)^k$, we complete the proof by Theorem 5.7 and Corollary 2.5.\end{proof}

\begin{rem} Many results concerning weighted EP and pseudo-weighted EP elements remain valid in the broader context of associative rings with an involution. We omit a detailed discussion here, as it closely mirrors the case for Banach algebras.\end{rem}

 {\bf Data Availability}: No datasets were generated or analyzed during the current
 study.\\

 {\bf Conflict of Interest}: The authors declare no competing interests.

\end{document}